\documentclass[twocolumn, amsthm]{autart}
\usepackage{mathptmx} 
\usepackage{stfloats}
\usepackage{comment}
\usepackage{arydshln}
\usepackage{amsmath}
\usepackage{amssymb}
\usepackage{mathtools}
\usepackage{graphicx}
\graphicspath{ {./images/} }

\newcommand{\sgn}{\operatorname{sgn}}

\newcommand{\maxapprox}{\text{max}_\delta}

\newcommand{\ttha}{\tilde{\theta}^{\rm a}} 
\newcommand{\gja}{G_J^{\rm a}}
\newcommand{\eja}{\eta_J^{\rm a}}
\newcommand{\gha}{G_h^{\rm a}}
\newcommand{\eha}{\eta_h^{\rm a}}
\newcommand{\ga}{\gamma^{\rm a}}
 
\newcommand{\gjaT}{G_J^{{\rm a} T}}

\newcommand{\tthae}{\tilde{\theta}^{\rm a, e}} 
\newcommand{\gjae}{G_J^{\rm a, e}}
\newcommand{\ejae}{\eta_J^{\rm a, e}}
\newcommand{\ghae}{G_h^{\rm a, e}}
\newcommand{\ehae}{\eta_h^{\rm a, e}}
\newcommand{\gae}{\gamma^{\rm a, e}}
\newcommand{\tthaeT}{\tilde{\theta}^{{\rm a, e} T}} 
\newcommand{\gjaeT}{G_J^{{\rm a, e} T}}

\newcommand{\tthac}{\tilde{\theta}_{{\rm c}}^{\rm a}} 
\newcommand{\gjac}{G_{J,{\rm c}}^{\rm a}}
\newcommand{\ejac}{\eta_{J,{\rm c}}^{\rm a}}
\newcommand{\ghac}{G_{h,{\rm c}}^{\rm a}}
\newcommand{\ehac}{\eta_{h,{\rm c}}^{\rm a}}
\newcommand{\gac}{\gamma_{\rm c}^{\rm a}}
\newcommand{\tthacT}{\tilde{\theta}_{{\rm c}}^{{\rm a} T}}

\newcommand{\omegaf}{\omega_{\textup{f}}}

\newtheorem{theorem}{Theorem}

\newtheorem{lemma}{Lemma}
\newtheorem{proposition}{Proposition}
\newtheorem{corollary}{Corollary}

\newtheorem{assumption}{Assumption}

\begin{document}

\begin{frontmatter}
\title{Local 
Practical Safe Extremum Seeking\\ 
with Assignable Rate of Attractivity to the Safe Set%
\thanksref{footnoteinfo}} 

\thanks[footnoteinfo]{Supported by Los Alamos National Lab LDRD DR
Project 20220074DR. Corresponding author A.~Williams.}

\author[LANL,UCSD]{Alan Williams}\ead{awilliams@lanl.gov}\ead{awilliam@ucsd.edu}, 
\author[UCSD]{Miroslav Krstic}\ead{krstic@ucsd.edu}, 
\author[LANL]{Alexander Scheinker}\ead{ascheink@lanl.gov}

\address[LANL]{Los Alamos National Lab, Los Alamos NM}  
\address[UCSD]{University of California - San Diego, San Diego CA}             


\begin{abstract}                          
We present Assignably Safe Extremum Seeking (ASfES), an algorithm designed to minimize a measured objective function while maintaining a measured metric of safety (a control barrier function or CBF) be  positive in a practical sense. We  ensure that for trajectories with safe initial conditions, the violation of safety can be made arbitrarily small with appropriately chosen design constants. We also guarantee an assignable ``attractivity'' rate: from unsafe initial conditions, the trajectories approach the safe set, in the sense of the measured CBF, at a rate no slower than a user-assigned rate. Similarly, from safe initial conditions, the trajectories approach the unsafe set, in the sense of the CBF, no faster than the assigned attractivity rate. The feature of assignable attractivity is not present in the semiglobal version of safe extremum seeking, where the semiglobality of convergence is achieved by slowing the adaptation. We also demonstrate local convergence of the parameter to a neighborhood of the  minimum of the objective function constrained to the safe set. The ASfES algorithm and analysis are multivariable, but we also extend the algorithm to a Newton-Based ASfES scheme (NB-ASfES) which we show is only useful in the scalar case. The proven properties of the designs are illustrated through simulation examples.
\end{abstract}

\end{frontmatter}

\section{Introduction}
We introduce an algorithm for Assignably Safe Extremum Seeking (ASfES), where both a lower bound on the rate of convergence {\em to} the safe set and an upper bound on the rate of convergence {\em towards} the unsafe set are user-assignable. 

The algorithm combines traditional Extremum Seeking (ES) methods \cite{krstic2000stability} \cite{ariyur2003real} with the quadratic program (QP)-based control barrier function (CBF) safety approach of \cite{ames2016control}. This work is essentially concerned with finding an algorithm which solves the following constrained optimization problem:
\begin{equation}
\min_{\theta(t)} J(\theta(t)) \text{ s.t. } h(\theta(t)) \geq 0 \text{ for all } t \in [0, \infty) . \label{eqn:opt_problem_intro}
\end{equation}
The objective function $J$ should be minimized but safety, represented by an unknown, yet measured, function $h(\theta)$, should be maintained by keeping $h$ positive over the entire course of time. The idea of ``practical safety'' has been described through the use of an inequality,
\begin{equation} \label{eqn:prac_safety_ineq_intro}
    h(\theta(t)) \geq h(\theta(0)) e^{-c t} + O(\epsilon) \text{ for all } t \in [0, \infty),
\end{equation}
where $c>0$ is value we term the ``attractivity'' rate and $O(\epsilon)$ is a violation of safety (in the worst case $O(\epsilon)$ is negative) which may be made small through appropriate choices of design constants. This inequality and the value $c$ describes how quickly a trajectory, which starts safe --- $\theta(0) \in \{h(\theta) \geq 0\}$ --- is allowed to approach the practical, marginally safe boundary $h=O(\epsilon)$. But, for trajectories which start unsafe --- $\theta(0) \in \{h(\theta) < 0\}$ --- it describes how fast (at least) a trajectory converges to the safe set. It is therefore desirable to be able to assign the value $c$. In our previous work \cite{williams2023semi}, the value $c$ is fundamentally unknown to the user, and not assignable. Because of the semiglobal design in \cite{williams2023semi}, the parameter $\theta(t)$ must travel \emph{slowly}, in order to achieve semiglobality, with a small adaptation gain. With a small adaptation gain, the estimates of the unknown gradients of the optimization problem can be computed accurately such that convergence can be achieved from any initial condition --- but this sacrifices the assignability of the parameter $c$. 

The design presented in this work allows the user to assign the value of $c$, achieving practical safety with an assignable attractivity rate, giving more control by the user to specify the dynamics of $h$. In addition, we also demonstrate local convergence of the parameter $\theta$ to a neighborhood of the constrained minimizer of the objective function on the safe set.

\subsection{Literature (ES, constrained ES, CBF safety filters)} 
Extremum seeking (ES) has evolved significantly since its first formal stability proof \cite{krstic2000stability}, with developments including handling of maps with multiple extrema \cite{tan2009global}, Newton-based ES \cite{ghaffari2012multivariable}, stochastic ES \cite{liu2015stochastic}, and more \cite{guay2015time,nesic2012framework}. Further innovations address stabilizing time-varying systems \cite{scheinker2012minimum,scheinker2014extremum}, applications to PDEs \cite{oliveira2016extremum}, fixed-time Nash equilibrium for networks \cite{poveda2022fixed}, among others \cite{scheinker2024100}. ES has also been used in situations with known constraints \cite{poveda2015shahshahani,tan2013extremum,scheinker2014extremum}.

Several works have explored ES-based solutions in constrained optimization with and an unknown constraint. The work \cite{guay2015constrained} modifies ES to estimate gradients of unknown functions, bypassing traditional averaging and singular perturbation theory. The Lie bracket formalism has been used \cite{labar2019constrained} to perform constrained ES approach, utilizing saddle point dynamics and an additional Lagrange-multiplier state. ES has been merged with ``boundary tracing" \cite{liao2019constrained}, drawing on nonlinear programming techniques. Authors in \cite{poveda2015shahshahani} present ES algorithms inspired by evolutionary game theory, addressing both equality and inequality constraints. Data-sampled ES techniques, such as those in \cite{hazeleger2022sampled} for steady-state constraint satisfaction and SPA stability, along with general optimization schemes explored in \cite{khong2013unified,teel2001solving}, highlight notable advances in unknown constraint handling. Previous work \cite{williams2023semi}, describes a semiglobal practically safe ES controller, with applications to particle accelerators \cite{williams2023experimental}, albeit with limited control over attractivity due to its semiglobal nature. 

First defined in \cite{Wieland} and later refined and popularized by the seminal papers \cite{AmesAutomotive,AmesCruiseControl}, QP-CBF formulations are often employed in a ``safety filter'' framework where they are used for generating safe control overrides for a potentially unsafe nominal controller. The QP-CBF methodology has found numerous application in control, \cite{WangMagnusMultiRobot,SantilloMulti,AmesCruiseControl,DSCC,CLARK2021Stochastic} to name only a small fraction of the body of work. Recent studies have explored safe set violations by disturbance magnitudes \cite{kolathaya2018input}. Inverse optimal safety filters for disturbances of both deterministic and stochastic natures have been designed \cite{krstic2023inverse}, without using the QP approach. Additionally, advancements in high-relative-degree CBFs \cite{krstic2006nonovershooting,hsu2015control,wu2015safety}, higher-order CBFs \cite{xiao2021bridging,krstic2006nonovershooting}, prescribed-time safety filters \cite{10288378}, and approximate optimal controllers integrating safety violations \cite{cohen2020approximate}, highlight diverse approaches to enhancing system safety and optimality.

\subsection{Results and Contribution}
The analysis we provide is based upon classical averaging and singular perturbation techniques. The QP-CBF based safety filters presented in \cite{ames2016control} are inherently nonsmooth, due to the term $\max \{x, 0\}$ appearing in the safety filter term. Therefore our design contains a smooth approximation of this quantity, $\maxapprox\{x\}$ --- a term which is inherently more conservative with respect to safety, and makes our algorithm more conservative than the design presented in \cite{williams2023semi}, yielding convergence to a neighborhood of the constrained equilibrium with a bias favoring safety (with conservativeness increasing with large $\delta$). Furthermore, our results are fundamentally local in nature relying on more restrictive assumption of $J$ and $h$ than those used in \cite{williams2023semi}, due to the analysis techniques used. The design also relies on a dynamic estimate of the quantity $||G_h||^{-2}$ by use of a Ricatti filter, reminiscent of \cite{ghaffari2012multivariable}. 

What is gained by the ASfES algorithm presented here over its semiglobal cousin \cite{williams2023semi} is that the user has enhanced control over the temporal behavior of $h$. Namely, we show that the value of $c$, the attractivity rate, in \eqref{eqn:prac_safety_ineq_intro} can be assigned by the user. We achieve: 1) local, practical asymptotic convergence to the constrained minimum of the objective function on the safe set 2) practical convergence to the safe set at least as fast as the assignable attractivity, for initial conditions which start unsafe 3) an approach towards the unsafe set, practically, no faster than the assigned attractivity for initial conditions which start safe.

Furthermore, we also present an extension to a Newton-Based ASfES (NB-ASfES) scheme which we show should only be used when considering a scalar parameter. NB-ASfES shares all of the features of ASfES with the additional feature that the nominal convergence rate of the parameter, $k$, is also assigned. Therefore we achieve an assignable rate of convergence of the parameter only if the assigned safety attractivity is not violated --- otherwise the parameter rate is slowed from its assigned rate in favor of maintaining the assigned safe attractivity rate.

This paper elucidates a tradeoff: one can either use slow adaptation and achieve semiglobality  \cite{williams2023semi} or, as shown here, assign the attractivity rate locally. 

The conference version of this work \cite{williams2022practically} studied the ASfES algorithm (see \cite{williams2022practically} for a intuitive derivation of the dynamics), but with analysis only given in 1 dimension. This work presents an extensive $n$ dimensional analysis. Other work in \cite{williams2023semi} studies a similar algorithm, which is semiglobal in nature, but has nonassignable attractivity. Practical safety is a feature shared by ASfES and the algorithm in \cite{williams2023semi}. Namely, that for an arbitrarily small violation of safety (the $O(\epsilon)$ term in \eqref{eqn:prac_safety_ineq_intro}), there exist design constants which guarantee it \emph{for all time}.

Organization: first we introduce the ASfES design and the main results in Sections \ref{sec:algo} and \ref{sec:main_results}, then we perform the averaging and convergence analysis in $n$ dimensions in Section \ref{sec:convergence}. In Section \ref{sec:safety_in_the_average}, the safety of the system is studied with the use of singular perturbation techniques, showing that the attractivity rate, $c$, can be assigned a prior by the user. Finally, Section \ref{sec:simulations} presents simulations of the algorithms for $1$ and $2$ dimensions.

\section{Algorithm} \label{sec:algo}
\subsection{Assignably Safe Extremum Seeking}
We will first introduce the set of differential equations describing the ASfES algorithm in $n$ dimensions:
\begin{align}
    \dot{\hat \theta} &= -k G_J + \gamma \maxapprox\{k G_J^T G_h - c \eta_h\} G_h \label{eqn:th_dyn}\\
    \dot G_J &= -\omegaf G_J + \omegaf(J( \hat \theta(t) + S(t) ) - \eta_J)M(t) \label{eqn:gj_dyn}\\
    \dot \eta_J &= -\omegaf \eta_J + \omegaf J( \hat \theta(t) + S(t) ) \label{eqn:etaj_dyn}\\
    \dot G_h &= -\omegaf G_h + \omegaf(h( \hat \theta(t) + S(t) ) - \eta_h)M(t) \label{eqn:gh_dyn}\\
    \dot \eta_h &= -\omegaf \eta_h + \omegaf h( \hat \theta(t) + S(t) ) \label{eqn:etah_dyn} \\
    \dot \gamma &= \omegaf  \gamma (1 - \gamma ||G_h||^2) \label{eqn:gamma_dyn}
\end{align}
where the state variables $\hat \theta, G_J, G_h \in \mathbb{R}^n$ and $\eta_J, \eta_h, \gamma \in \mathbb{R}$ so overall the dimension of the system is $3n+3$. The map is evaluated at $\theta$, defined by 
\begin{equation}
    \theta(t) = \hat \theta(t) + S(t) \; .
\end{equation}
The integer $n$ denotes the number of parameters one wishes to optimize over. The design coefficients are $k, c, \delta, \omegaf \in \mathbb{R}_{>0}$. The perturbation signal $S$ and demodulation signal $M$ are given by $S_i(t) = a \sin(\omega_i t)$ and $M_i(t) = \frac{2}{a} \sin(\omega_i t)$
the standard signals found in classical ES \cite{ariyur2003real}. The functions $J, h: \mathbb{R}^n \mapsto \mathbb{R}$ are scalar valued functions defined over the optimization variables $\theta$. 
To make the classical averaging theorem applicable, the algorithm must have a continuously differential right-hand side. Because the standard maximum operation is nonsmooth in the QP-CBF formulation we approximate it by the function $\maxapprox:\mathbb{R} \mapsto \mathbb{R}_{>0}$ given by
\begin{equation}\label{eq-maxdelta}
    \maxapprox\{x\} := \frac{1}{2}\left(x + \sqrt{x^2 + \delta}\right) \approx \max\{x,0\},
\end{equation}
for a small parameter $\delta>0$. Note that $\maxapprox\{x\} > \max\{x,0\}$ for all $x \in \mathbb{R}$. 

\begin{figure*}[t]
  \centering
  \includegraphics[bb=0 0 500 200, scale=0.75]{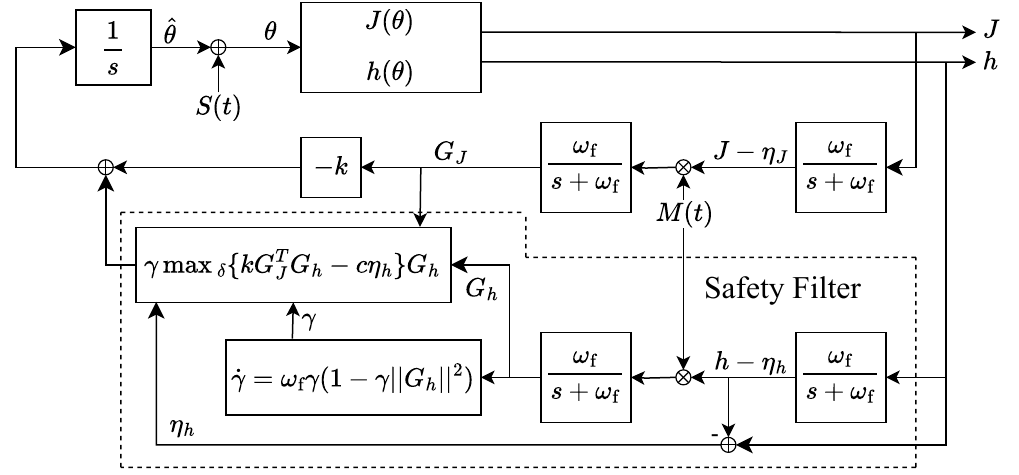}
  \caption{A block diagram of the ASfES algorithm. Removing the ``Safety Filter'' block recovers the classical extremum seeking algorithm.}
  \label{fig:block_diagram_ASfES}
\end{figure*}

\subsection{Notation and Coordinate Transformations}
The variable $\theta$, is transformed several times before averaging is performed, and then again to the equilibrium of the constrained system. We summarize the various coordinates in Table \ref{table:theta_transformations} for clarity, using ``param." to mean ``parameter", ``est." to mean ``estimate" or ``estimation", ``err." to mean ``error", and ``avg." to mean ``average".

\begin{table}[h!]
\centering
\caption{Transformations of the parameter $\theta$.}
\def\arraystretch{1.2}
\setlength{\tabcolsep}{3pt}
\begin{tabular}{||r l l||}
 \hline
 Var. & Definition &  Description \\ 
 [0.5 ex] 
 \hline\hline
 $\theta(t)$ & n/a  &  optimization param. (input to the map) \\
 $\hat \theta(t)$ & $ = \theta(t) - S(t)$  &  unconstrained param. est. \\
 $\tilde \theta(t)$ & $ =  \hat\theta(t) -\theta^*$ & unconstrained param. est. err. \\
 $\tilde \theta^{\rm a} (t)$ & avg. of $\tilde \theta(t)$ &  unconstrained avg. param est. err. \\
 $\tilde \theta^{\rm a, e}$ & eqm. of $\tilde \theta^{\rm a} (t)$ & eqm. of unconstrained avg. param. est. err. \\
 $\tilde \theta^{\rm a}_{\rm c} (t)$ & $ =  \tilde \theta^{\rm a} (t) -\tilde \theta^{\rm a, e}$ &  constrained avg. param. est. err. \\ [1ex] 
 \hline
\end{tabular}
\label{table:theta_transformations}
\end{table}

In unconstrained extremum seeking, the equilibrium $\tilde \theta^{\rm a,e} = 0$. This is not true for the dynamics we present. The final transformation in Table \ref{table:theta_transformations} defines a new state with equilibrium of zero $\tilde \theta^{\rm a,e}_{\rm c} = 0$. The rest of the states in the algorithm, $G_J$, $\eta_J$, $G_h$, $\eta_h$, and $\gamma$, undergo fewer transformations than $\theta$. 

We use the variable `$x$' to describe the full state of the solutions or equilibrium:
\begin{align}
    x &= [\theta; G_J; \eta_J; G_h; \eta_h; \gamma] ,\\
    \hat x &= [\hat \theta; G_J; \eta_J; G_h; \eta_h; \gamma] ,\\
    \tilde x &= [\tilde \theta; G_J; \eta_J; G_h; \eta_h; \gamma] ,\\
    x^{{\rm a}} &= [\tilde \theta^{\rm a}; G^{\rm a}_J; \eta^{\rm a}_J; G^{\rm a}_h; \eta^{\rm a}_h; \gamma^{\rm a}] ,\\
    x^{{\rm a, e}} &= [\tilde \theta^{\rm a, e}; G^{\rm a, e}_{J}; \eta^{\rm a, e}_{J}; G^{\rm a, e}_{h}; \eta^{\rm a, e}_{h}; \gamma^{\rm a, e}] , \\
    x^{{\rm a}}_c &= [\tthac ; \gjac ; \ehac ; \ghac ; \gac ; \ejac ] .
\end{align}
Notation: the superscript ``a'' always denotes a variable defined from an averaged system, and the superscript ``e'' always denotes an equilibrium value. For example, $x^{\rm a, e}$ is the equilibrium of the state $x^{\rm a}$, which is the averaged version of $\tilde x$ (shown in the next section). Here, ``averaged'' is in the sense in the standard averaging theory \cite{khalil}. The subscript ``c'' relate to the ``constrained'' dynamics of ASfES, shown later to be defined as $x^{\rm a}_c :=  x^{\rm a} -  x^{{\rm a, e}}$. 
Note that the ordering of the state variables have also been changed in $x^{\rm a}_c$, which will help identify block structures in a linearization procedure later on. Also, when stacking vectors $x\in \mathbb{R}^n$ and $y \in \mathbb{R}^m$ we denote $z = [x;y] := [x_1,...,x_n,y_1,...,y_m]^T \in \mathbb{R}^{n+m}$. We define the signum function as: $\sgn(x) := -1 \text{ for } x < 0, 0 \text{ for } x = 0, 1 \text{ for } x > 0$. We define the unit step function as: $u(x) := 0 \text{ for } x < 0, 0.5 \text{ for } x=0, 1 \text{ for } x > 0$.

\section{Main Results} \label{sec:main_results}
We consider the analysis of ASfES under the following assumptions.
\begin{assumption} \label{assum:1}
With the symmetric Hessian matrix $H \succ 0$ and unconstrained minimum $\theta^* \in \mathbb{R}^n$, the objective function is unknown and has the quadratic form 
\begin{equation}\label{eq-Jinto}
    J(\theta) = J^* + \frac{1}{2}(\theta - \theta^*)^T H (\theta - \theta^*).
\end{equation} 
\end{assumption}
\begin{assumption} \label{assum:2}
With $h_0 \in \mathbb{R}$ and $h_1 \in \mathbb{R}^n_{\neq 0}$, the barrier function is unknown and has the linear form 
\begin{equation}\label{eq-bintro}
    h(\theta) = h_0 + h_1^T (\theta - \theta^*).
\end{equation} 
\end{assumption}

Since the results are local, a linear form of the barrier function is assumed essentially with no loss of generality.
In \eqref{eq-bintro}, $h_0 \geq 0$ implies the unconstrained minimum of $J$ is safe, and $h_0<0$ implies the unconstrained minimum is unsafe.

\begin{assumption} \label{assum:3}
The design constants are chosen as $\omegaf, \omega_i, \delta,a, k,c > 0$, where $\omega_i \slash \omega_j$ are rational with frequencies $\omega_i$ chosen such that $\omega_i\neq \omega_j$ and $\omega_i + \omega_j \neq \omega_k$ for distinct $i, j,$ and $k$.
\end{assumption}

The first main result of this paper concerns convergence of the parameter to the (constrained) minimizer of the objective function $J$ on the safe set.
\begin{theorem} [Local Convergence to the Safe Optimum] \label{thm:convergence_avg}
Under Assumptions \ref{assum:1}--\ref{assum:3}, there exists positive constants $\rho, \omega^*$ such that if $||\tilde{x}(0) - x^{{\rm a, e}}|| < \rho$ 
then for all $\omega \in [\omega^*, \infty)$,
\begin{equation}
    \limsup_{t\rightarrow\infty} J(\theta(t)) = J_s^* + O(1/\omega + a + \delta),
\end{equation}
where $J_s^*$ is the minimum of $J(\theta)$ on the safe set $\mathcal{C} = \{\theta: h(\theta) \geq 0 \}$.
\end{theorem}
The local region $||\tilde{x}(0) - x^{{\rm a, e}}|| < \rho$ describes initial conditions which are close to the constrained minimizer. This implies $\hat \theta(0)$ must be close to $\theta^* + \tthae$ --- a quantity which is the global minimizer ($\theta^*$) and a large offset, describing the shift of the parameter into the safe set. Additionally the filtered estimates $G_J, \eta_J, G_h, \eta_h, \gamma$ must also lie close to the true values.

In the second main result, we show that the violation of safety of the original system is of order $O(1/\omega + 1/\omegaf + a)$ and the decay of the time dependent term has a time constant $c$, which is specified by the user and contained in the algorithm dynamics \eqref{eqn:th_dyn} - \eqref{eqn:gamma_dyn}. 

\begin{theorem} [Practical Safety with Assignable Attractivity Rate] \label{thm:assignable_prac_safety}
Under Assumptions \ref{assum:1}--\ref{assum:3}, there exists positive constants $\rho, \omegaf^*, \omega^*$ such that if $||\tilde{x}(0) - x^{{\rm a, e}}|| < \rho$, then for all $\omegaf \in [\omegaf^*, \infty)$ and $\omega \in [\omega^*, \infty)$,
\begin{align}\label{eq-safehthetac}
    h(\theta(t)) &\geq h(\theta(0)) {\rm e}^{-ct} + O(1/\omegaf + 1/\omega +a).
\end{align}
\end{theorem}

\section{Convergence} \label{sec:convergence}
The next subsections present the main steps in deriving Theorem \ref{thm:convergence_avg}. First an average system is derived, with a unique equilibrium, and then linearization is performed to conclude the convergence result.

\subsection{Deriving the Average System} \label{sec:averaging}
With the change of variables, introducing $\hat \theta$ and $\tilde \theta$, we have the relations 
\begin{align}
\tilde\theta(t) &=  \hat\theta(t) -\theta^*\\
\theta(t) &=  \hat\theta(t) + S(t) = \tilde\theta(t) + S(t) +\theta^*\\
\theta(t)-\theta^* &=  \tilde\theta(t) +S(t)
\end{align}
where 
\begin{equation}
{S}_i(t) = a \sin(\omega_i t),
    \end{equation}
and with the time transformation $\tau = \omega t$ we have
\begin{align}
    \omega \frac{d \tilde{x}}{d\tau}    
    &=
    \begin{bmatrix*}[l]
     -k G_J + \gamma \maxapprox\{k G_J^T G_h - c \eta_h\} G_h \\
    -\omegaf G_J + \\
    \quad \omegaf(J(\tilde \theta + S(\frac{\tau}{\omega}) +\theta^*) - \eta_J)M(\frac{\tau}{\omega})  \\
    -\omegaf \eta_J + \omegaf J(\tilde \theta + S(\frac{\tau}{\omega}) +\theta^*) \\
    -\omegaf G_h + \\
    \quad \omegaf(h(\tilde \theta + S(\frac{\tau}{\omega}) +\theta^*) - \eta_h)M(\frac{\tau}{\omega})  \\
    -\omegaf \eta_h + \omegaf h(\tilde \theta + S(\frac{\tau}{\omega}) +\theta^*) \\
    \omegaf  \gamma (1 - \gamma ||G_h||^2)
    \end{bmatrix*} \\
    &=f(\tilde x, \tau). \label{eqn:sys_transformed}
\end{align}
Now the system is in the correct form to perform averaging. We express the the signals $S$ and $M$ with frequencies scaled by $\omega$ as  $S_i (\frac{\tau}{\omega}) = a \sin(\omega_i' \tau)$ and $M_i(\tau) = \frac{2}{a} \sin(\omega_i' \tau)$, with $\omega_i := \omega \omega_i'$. We compute the average of the right-hand side  given by the formula \eqref{eqn:sys_transformed}, taking the period to be $\Pi$ given as
 
\begin{equation}
    \Pi = 2 \pi \times \text{LCM} \left\{   \frac{1}{\omega_i'}\right\}, \quad i \in \{1,2,...,n\},
\end{equation}
where LCM denotes the least common multiple. We compute the average system as
\begin{align}
   f^{\textup{a}}(x^{\rm a}) &:= \frac{1}{\Pi} \int_{0}^{\Pi} f(x^{\rm a}, \tau) d \tau, \label{eqn:avg_general} \\
   x^{{\rm a}} &:= [\tilde \theta^{\rm a}; G^{\rm a}_J; \eta^{\rm a}_J; G^{\rm a}_h; \eta^{\rm a}_h; \gamma^{\rm a}].
\end{align}
Using classical averaging \cite{khalil}, we arrive at the average system
\begin{align}
    \omega \frac{d x^{\rm a}}{d\tau} 
    &=
    \begin{bmatrix*}[l]
     -k \gja + \ga \maxapprox\{k \gjaT \gha - c \eha\} \gha \\
    -\omegaf \gja + \omegaf H \ttha  \\
    -\omegaf \eja + \omegaf (J(\ttha + \theta^*) + \frac{a^2}{4} \text{Tr}(H)) \\
    -\omegaf \gha + \omegaf h_1  \\
    -\omegaf \eha + \omegaf h(\ttha + \theta^*) \\
    \omegaf  \ga (1 - \ga ||\gha||^2)
    \end{bmatrix*}\\ 
    &=  f^{\textup{a}}(x^{\rm a}). \label{eqn:avg_sys}
\end{align}
\begin{figure}[t]
  \centering
  \includegraphics[width=\columnwidth]{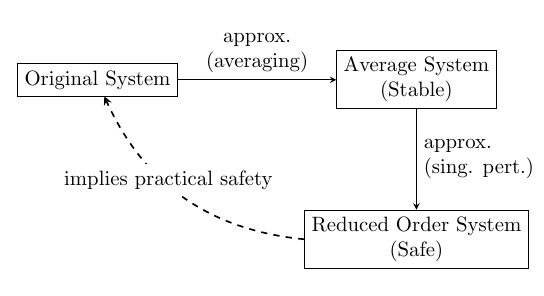}
  \caption{Road map of the analysis steps for ASfES. The approximation between the original system and average is performed via linearization of the average system (Section \ref{sec:convergence}) and the approximation between the average system and the reduced order system is performed via singular perturbation (Section \ref{sec:safety_in_the_average}) by taking $\omegaf$ large.}
  \label{fig:analysis_roadmap}
\end{figure}
\subsection{Equilibrium of the Average System} \label{sec:eqm}
We find the equilibrium of the average system \eqref{eqn:avg_sys}, taking the positive equilibrium point of $\ga$. Immediately we calculate the following equilibria 
\begin{align}
    \ejae  & = J^* + \frac{1}{2} \tthaeT H \tthae + \frac{a^2}{4} \text{Tr}(H) \\
    \ghae  &= h_1 \\
    \gae   &= \frac{1}{|| h_1 ||^2}. \label{eqn:eqm_gamma1}
\end{align}
for three components of the system state. And we have the following relations:
\begin{align}
     0 & = -k \gjae + \gae \maxapprox\{k \gjaeT \ghae - c \ehae\} \ghae \label{eqn:eqm1}\\
    \gjae &= H \tthae \\
    \tthae &= H^{-1}\gjae \\
    \ehae & = h(\theta^* + \tthae) = h_0 + h_1^T \tthae .
\end{align}
From \eqref{eqn:eqm1} we have the following quadratic vector equation in $\gjae$
\begin{align}
    k || h_1 ||^2 \gjae &= h_1 \maxapprox(k \gjaeT h_1 - c h_0 - c h_1^T H^{-1} \gjae). \label{eqn:quad_GJ_1}
\end{align}

To solve \eqref{eqn:quad_GJ_1}, we use the following fact: for some $\nu > 0$ we have 
\begin{equation}
    k || h_1 ||^2 \gjae = h_1 \nu
\end{equation}
because $\maxapprox(x) > 0$. Writing $\gjae = h_1 \frac{\nu}{k ||h_1||^2}$, and using the definition of $\maxapprox$, we achieve the expansion of \eqref{eqn:quad_GJ_1}:
\begin{align}
    h_1 \nu  &= h_1 \frac{1}{2} (\nu- c h_0 - d \nu) + \nonumber \\ & \qquad h_1  \frac{1}{2} \sqrt{(\nu - c h_0 - d \nu)^2 + \delta}
\end{align}
letting the quantity $d>0$ (because $H\succ0$) be
\begin{equation}
    d = \frac{c}{k ||h_1||^2} h_1^T H^{-1} h_1.
\end{equation}
After solving for the positive solution of the quadratic equation in $\nu$, the (unique) equilibrium of the average system can be written as below
\begin{align}
    \tthae &= H^{-1}\gjae \label{eqn:eqm_theta} , \\
    \gjae  &= \frac{|h_0|}{2 h_1^T H^{-1} h_1} \left(-\sgn(h_0) + \sqrt{1 + \delta \frac{d}{c^2 h_0^2}} \right) h_1 , \\
    \begin{split}
        \ejae &= J^* + \frac{a^2}{4} \text{Tr}(H) + \\ & \quad \frac{h_0^2}{8 h_1^T H^{-1} h_1} \left(-\sgn(h_0) + \sqrt{1 + \delta \frac{d}{c^2 h_0^2}} \right)^2 , 
    \end{split} \\
    \ghae  &= h_1, \\
    \ehae  & = \frac{|h_0|}{2}\left(\sgn(h_0)+  \sqrt{1 + \delta \frac{d}{c^2 h_0^2}}\right), \label{eqn:eqm_eta_h} \\
    \gae   &= \frac{1}{|| h_1 ||^2}, \label{eqn:eqm_gamma}.
\end{align}

We see that no matter the sign of $h_0$, this equilibrium always lies in the interior of the safe region because $\ehae = h(\theta^* + \tthae)> 0$. 
Also, we can think of $\ejae$ as the equilibrium value of the filtered measurement of the objective function. It settles to $J^*$ with a small error due to the dithering amplitude, and a possibly large offset if $h_0<0$, resulting from the safety dynamics constraining the equilibrium to the safe set.

\subsection{Linearization} \label{sec:linearization}
Using the in error variables of the average system (to the constrained equilibrium) we define the following coordinates with new ordering as
\begin{align}
x^{{\rm a}}_c(t) &:= [\ttha (t) - \tthae ; \gja (t) - \gjae; \eha (t) - \ehae; \nonumber \\ &\qquad \gha (t) - \ghae; \ga(t) - \gae; \eja (t) - \ejae] \nonumber \\
&= [\tthac (t); \gjac (t); \ehac (t); \ghac (t); \gac(t);  \ejac (t)], \label{eqn:avg_error_c}
\end{align}
and the system
\begin{equation}
    \omega \frac{d x^{\rm a}_c}{d\tau} = g(x^{\rm a}_c), \label{eqn:avg_error_c_dynamics}
\end{equation}
from \eqref{eqn:avg_sys}. We have changed the ordering of the dynamic equations to help identify block triangular structures and have therefore have defined the system under with a new function $g$, instead of $f^{\textup{a}}$. For example, $\omega \dot{G}_{h,{\rm c}}^{\rm a} = g_3(x^{\rm a}_c) = -\omegaf (\ghac + \ghae )+ \omegaf h_1$, with a slight abuse of notation as $g_i$ denotes the the vector valued function associated with the dynamics of the $i$th vectored valued components of $x^{{\rm a}}_c(t)$ listed in \eqref{eqn:avg_error_c}.

The Jacobian matrix of \eqref{eqn:avg_error_c_dynamics} which is $J = \frac{\partial g}{\partial x^{\rm a}_c}$, is a matrix of size ${(3n+3) \times (3n+3)}$ with structure:
\begin{align}
J &=
\left[ 
\arraycolsep=2.0pt\def\arraystretch{2.0} 
\begin{array}{c c c ;{2pt/2pt} c c ;{2pt/2pt} c }
0 & \dfrac{\partial g_1}{\partial \gjac} &\dfrac{\partial g_1}{\partial \ehac} &\dfrac{\partial g_1}{\partial \ghac} &\dfrac{\partial g_1}{\partial \gac} & 0 \\
\dfrac{\partial g_2}{\partial \tthac} & \dfrac{\partial g_2}{\partial \gjac} & 0 & 0 & 0 & 0\\
\dfrac{\partial g_3}{\partial \tthac} & 0 &\dfrac{\partial g_3}{\partial \ehac} & 0 & 0 & 0\\ \hdashline[2pt/2pt]
0 & 0 & 0 &\dfrac{\partial g_4}{\partial \ghac} & 0  & 0\\
0 & 0 & 0 &\dfrac{\partial g_5}{\partial \ghac} &\dfrac{\partial g_5}{\partial \gac} & 0 \\ \hdashline[2pt/2pt]
\dfrac{\partial g_6}{\partial \tthac} & 0 & 0 & 0 & 0 & \dfrac{\partial g_6}{\partial \ejac}
\end{array} \right] \label{J} 
\end{align}
Evaluating at the equilibrium yields
\begin{equation}
 \left. {J} \right|_{x^{\rm a}_c=0} =
\left[ \begin{array}{c c c }
\left. {J_{11}} \right|_{x^{\rm a}_c=0} & \left. {J_{12}} \right|_{x^{\rm a}_c=0} & 0\\
0 &  \left. {J_{22}} \right|_{x^{\rm a}_c=0} & 0  \\
\left. {J_{31}} \right|_{x^{\rm a}_c=0} & 0 &  \left. {J_{33}} \right|_{x^{\rm a}_c=0}
\end{array} \right]. \label{J_eval}
\end{equation}
The eigenvalues of this block triangular matrix are the eigenvalues of the diagonal blocks. Therefore we analyze $J_{11}$ of size $(2n+1) \times (2n+1)$, $J_{22}$ of size $(n+1)\times (n+1)$, and the scalar $J_{33}$ separately. We can easily compute 
\begin{equation}
    \left. {J_{33}}
 \right|_{x^{\rm a}_c=0} = \left. \dfrac{\partial g_6}{\partial \ejac} \right|_{x^{\rm a}_c=0} =-\omegaf
\end{equation}
and the block $J_{22}$ evaluated at the origin can
be written as
\begin{equation}
 \left. {J_{22}}
 \right|_{x^{\rm a}_c=0} =
 \left[
    \begin{array}{c c}
    - \omegaf I  & 0 \\
    l & - \omegaf
    \end{array} 
    \right] \label{J22}
\end{equation}
where $l$ need not be calculated. Because this matrix is itself block triangular, it's eigenvalues lie on the diagonal and are $\lambda_i = -\omegaf$ for $i \in \{1, 2, ..., n+1\}$, therefore \eqref{J22} is Hurwtiz. We have discovered $2n+1$ eigenvalues equal to $-\omegaf$.

After careful algebra and differentiation we express
\begin{equation}
\left. {J_{11}}
 \right|_{x^{\rm a}_c=0} =
    \begin{bmatrix}
    0 && -M && - c \alpha|| h_1 ||^{-2} h_1 \\
    \omegaf H && -\omegaf I  && 0 \\
    \omegaf h_1^T && 0 && -\omegaf
    \end{bmatrix}, \label{J11_}
\end{equation}
with the symmetric, positive definite matrix $M \succ 0$ and the scalar $0 < \alpha < 1$ defined as
\begin{align}
    M &= k \left( I -  \alpha  \frac{h_1 h_1^T}{|| h_1 ||^2} \right) , \label{eqn:M}\\
    \alpha &= \frac{1}{2} \left( \frac{k c_1|| h_1 ||^2 - c \ehae}{\sqrt{(k c_1|| h_1 ||^2 - c \ehae)^2 + \delta}} + 1\right), \\ \label{eqn:c3}
    c_1 &= \frac{|h_0|}{2 h_1^T H^{-1} h_1} \left(-\sgn(h_0) + \sqrt{1 + \delta \frac{d}{c^2 h_0^2}} \right).      
\end{align}

We provide an extensive analysis of matrices in the form of \eqref{J11_} in Lemma \ref{lem:all_properties_X} the appendix, but present the most essential result below.
\begin{corollary}
Under Assumptions \ref{assum:1}--\ref{assum:3} and for $\alpha \in [0,1]$, the matrix $\left. {J_{11}}
 \right|_{x^{\rm a}_c=0}$ in \eqref{J11_} is always Hurwtiz, and therefore the Jacobian of the constrained average error system $\left. {J_{11}}
 \right|_{x^{\rm a}_c=0}$ in \eqref{J_eval} is always Hurwtiz. \label{corol:hurwitz}
\end{corollary}

This result is slightly more general than what absolutely necessary for the convergence analysis because it gives results for $\alpha$ potentially being equal to zero or one, even though the assumption of $\delta>0$ precludes this possibility. This is still nonetheless useful knowledge because it tells us that as $\delta \to 0$, the eigenvalues of the linearization do not tend to undesirable values and provides hope for the future in developing a nonsmooth version of this algorithm with $\delta = 0$.

From Theorem 10.4 in \cite{khalil}, Corollary \ref{corol:hurwitz}, and Proposition \ref{prop:const_min} given in the appendix, we conclude the following results.

\begin{proposition} [Approximation by the Average System] \label{prop:closeness_avg}
Consider the exponentially stable equilibrium point $x^{{\rm a, e}}$ \eqref{eqn:eqm_theta}--\eqref{eqn:eqm_gamma} of the average system in \eqref{eqn:avg_sys}. Under Assumptions \ref{assum:1}--\ref{assum:3}, there exists positive constants $\rho, \omega^*$ such that if $||\tilde{x}(0) - x^{{\rm a, e}}|| < \rho$ 
then for all $\omega \in [\omega^*, \infty)$,
\begin{align} \label{eq-xav-xhat}
||x^{\rm a}(t) - \tilde{x}(t)|| &= O (1/\omega) \; \text{for all }t \in [0, \infty), \\
||\theta(t) - (\tilde{\theta}^{\rm a}(t) + \theta^*)|| &= O (1/\omega+a) \; \text{for all }t \in [0, \infty). 
\end{align}
\end{proposition}
The  equilibrium of the average system \eqref{eqn:avg_sys}, given in \eqref{eqn:eqm_theta}--\eqref{eqn:eqm_gamma}, resides in the interior of the safe set regardless of whether the unconstrained minimizer $\theta_{\rm min}$ of $J$ lies in the safe set or not. One can show this by evaluating $h$ in the original coorindates at the value $\tilde \theta^{\rm a, e}$. Additionally, as evident from \eqref{eqn:eqm_eta_h}, the parameter $\delta$ in the softened max operation, \eqref{eq-maxdelta}, creates a bias in the equilibrium that favors safety.

The proof of Theorem \ref{thm:convergence_avg} utilizes the results of Section \ref{sec:convergence}, and is shown in the appendix.

\section{Safety} \label{sec:safety_in_the_average}
This section provides a description of safety, through the use of singular perturbation results. The following results also motivates the addition the $\maxapprox$ safety term which differentiates ASfES from standard extremum seeking. In the previous section we showed that through a linearization of the average system, the dynamics of ASfES are locally stable in $n$ dimensions. We also showed that the average system and the original system are close. In the this section, we begin by showing that the parameter $\tilde \theta^{\rm a}(t)$ of the average system is close to a reduced order model, denoted by its parameter $\theta_{\rm r}(t)$ - this reduced model is created by taking the the filter gains $\omegaf \to \infty$. These results are used in the proof of Theorem \ref{thm:assignable_prac_safety} shown in the appendix. See Fig. \ref{fig:analysis_roadmap} for a visual depiction of the high level steps in the analysis.
  
Consider the singular perturbation as $\omegaf$ approaches $\infty$. We can express \eqref{eqn:avg_error_c_dynamics} in the form of the standard singular perturbation model with $\epsilon = \frac{1}{\omegaf}$ in the original time $t=\omega/\tau$ with the original ordering:
\begin{align}
    \dot{\tilde{\theta}}^{\rm a}_{\rm c} &= w(z), \label{eqn:sing_pert1} \\ 
    \epsilon \dot{z} &= p(\tthac,z), \label{eqn:sing_pert2}
\end{align}
where $z = [\gjac;\ejac;\ghac;\ehac;\gac]$ and 
\begin{multline}
w(z) = -k (\gjac+\gjae) + (\gac+\gae) (\ghac+\ghae) \maxapprox\{ \beta \},
\end{multline}
\begin{equation}
p(\tilde{\theta}^{\rm a},z) =
\begin{bmatrix*}[l]
&  -\omegaf (\gjac + \gjae) + \omegaf H (\tthac + \tthae)  \\
&  -\omegaf (\ejac + \ejae) + \omegaf (J(\tthac + \tthae + \theta^*) + \\ & \qquad\frac{a^2}{4} \text{Tr}(H)) \\
&  -\omegaf (\ghac + \ghae) + \omegaf h_1  \\
&  -\omegaf (\ehac + \ehae) + \omegaf h_0 + \omega_f h_1 (\tthac + \tthae + \theta^*) \\
&  \omegaf  (\gac + \gae) (1 - (\gac + \gae) ||\ghac + \ghae||^2)
\end{bmatrix*},
\end{equation}
with the quantity 
\begin{equation}
    \beta = k (\gjac + \gjae)^T (\ghac + \ghae) - c (\ehac + \ehae)
\end{equation}
In an effort to meet the conditions of Theorem 11.2 `Singular Perturbation on the Infinite Interval' in \cite{khalil}, we must show that the reduced model and the boundary layer model are exponentially stable.
Let us denote the root of $0 = p(\tthac,z)$ as the quasi-steady state 
\begin{equation}\label{eq-qss}
z = 
\begin{bmatrix*}[c]
& \gjac \\
& \ejac \\
& \ghac \\
& \ehac \\
& \gac
\end{bmatrix*}
= r(\tthac) = 
\begin{bmatrix*}[l]
&  H \tthac \\
&  \frac{1}{2} \tthacT H \tthac\\
&  0 \\
&  h_1^T \tthac \\
&  0
\end{bmatrix*}.
\end{equation}
The reduced model is $n$-dimensional differential equation
\begin{equation} \label{eqn:reduced}
    \dot{\theta}_{\rm r} = w(r(\theta_{\rm r})).
\end{equation}
To show stability of the reduced order model we must determine if the Jacobian 
\begin{equation} \label{eqn:reduced_lin}
J_r = \left. \frac{\partial w(r(\theta_{\rm r}))}{\partial \theta_{\rm r}}
 \right|_{\theta_{\rm r}=0} 
\end{equation}
is Hurwtiz. We are able to use the calculations in Section \ref{sec:linearization} to help us differentiate, using the terms 
\begin{align}
    \left. \dfrac{\partial g_1(x^{\rm a}_c)}{\partial \gjac} \right|_{x^{\rm a}_c=0} &= -M \label{eqn:first_partial}\\
    \left. \dfrac{\partial g_1(x^{\rm a}_c)}{\partial \ehac} \right|_{x^{\rm a}_c=0} &= -c \alpha ||h_1||^{-2} h_1 \label{eqn:second_partial}
\end{align}
from \eqref{J11_}. Equation \eqref{eqn:first_partial} is simply $\frac{\partial w(z)}{\partial \gjac}$ and \eqref{eqn:second_partial} is $\frac{\partial w(z)}{\partial \ehac}$, both evaluated at $z=0$. Now consider the following boundary layer variable as functions of $\theta_{\rm r}$ as follows: $\gjac = \gjac(\theta_{\rm r}) =  H \theta_{\rm r}$ and $\ehac = \ehac(\theta_{\rm r}) = h_1^T \theta_{\rm r}$. Therefore we can write the RHS of \eqref{eqn:reduced} as
\begin{align}
    w(r(\theta_{\rm r})) &=  -k (\gjac(\theta_{\rm r})+\gjae) + (\gac+\gae) (\ghac+\ghae) \nonumber \\&\qquad \maxapprox\{ \beta \}, \\
    \beta &= k (\gjac(\theta_{\rm r}) + \gjae)^T (\ghac + \ghae) - \nonumber \\&\qquad c (\ehac(\theta_{\rm r}) + \ehae)
\end{align}
where we have yet to use the relations in \eqref{eq-qss}. Computing the Jacobian, and linearizing \eqref{eqn:reduced}, we can write \eqref{eqn:reduced_lin} as
\begin{align}
J_r&=\left.\left(\frac{\partial w(r(\theta_{\rm r}))}{\partial \gjac} \frac{\partial \gjac}{\partial \theta_{\rm r}} + \frac{\partial w(r(\theta_{\rm r}))}{\ehac} \frac{\ehac}{\partial \theta_{\rm r}} \right) \right|_{\theta_{\rm r}=0} \\
&= \left.\frac{\partial w(r(\theta_{\rm r}))}{\partial \gjac}\right|_{\theta_{\rm r}=0} \frac{\partial \gjac}{\partial \theta_{\rm r}} + \left.\frac{\partial w(r(\theta_{\rm r}))}{\ehac} \right|_{\theta_{\rm r}=0} \frac{\ehac}{\partial \theta_{\rm r}}  \\
&= -M H - c \alpha \frac{h_1 h_1^T}{||h_1||^2}
\end{align}
By Lemma \ref{lem:eig_real_pos} we conclude the eigenvalues of $J_r$ are real and negative and the reduced model is exponentially stable at the origin.
 
The boundary layer model is
\begin{equation}
    \dot{z}_{\rm b} = p(\theta, z_{\rm b} + r(\theta))
\end{equation}
for a fixed $\theta$. The boundary layer model can be shown to be 
\begin{equation} \label{eqn:boundary_layer}
    \dot{z}_{\rm b} = \begin{bmatrix}
    -z_{\rm b,1} \\
    -z_{\rm b,2} \\
    -z_{\rm b,3} \\
    -z_{\rm b,4} \\
    (z_{\rm b,5} + \frac{1}{||h_1||^2})(1-(z_{\rm b,5} + \frac{1}{||h_1||^2})||z_{\rm b,3} + h_1||^2)
    \end{bmatrix}.
\end{equation}
with a slight notation abuse denoting $z_{\rm b,i}$ as either a vector or scalar variable corresponding with the vectors and scalars $\gjac, \ejac ,$ ... $,\gac$. It is evident by inspection that the linearization of \eqref{eqn:boundary_layer} at the origin has $2n+3$ eigenvalues at $-1$. (Note that the linearization of the last row will yield a cross term due to $z_{\rm b,3}$, but since the linearization is triangular we need only consider the diagonal entrees.) Hence, the boundary layer model is exponentially stable. 

The reduced order model in \eqref{eqn:reduced2} 
\begin{align} \label{eqn:reduced2}
    \dot{\theta}_{\rm r} &= -k(H \theta_{\rm r} + \gjae) + \frac{1}{||h_1||^2} \maxapprox\{k (H \theta_{\rm r} + \gjae)^T h_1  \nonumber \\ &\qquad - c(h_1^T\theta_{\rm r} + \ehae)\} ,
\end{align}
can be expressed in the coordinates $\theta_{\rm r} = \tilde {\theta}_{\rm r} - \tthae$, as
\begin{equation} \label{eqn:reduced3}
    \dot{\tilde \theta}_{\rm r} = -k H \tilde {\theta}_{\rm r} + \frac{h_1}{||h_1||^2} \maxapprox\{k\tilde {\theta}_{\rm r}^T H  h_1 - c(h_0 + h_1^T\tilde {\theta}_{\rm r}) \} ,
\end{equation} 
using the relations $\ehae=h_0 + h_1^T H^{-1} \gjaeT$ and $\tthae = H^{-1} \gjae$. We can now state the following result having satisfied the conditions in \cite{khalil}. 

\begin{proposition} [Singular Perturbation of Average System] \label{prop:sing_pert}
Let the solution of \eqref{eqn:avg_sys} be given by $x^{\rm a}(t)$, with its first component ${\tilde \theta}^{\rm a}(t)$ in the time scale $t$, the solution of reduced model \eqref{eqn:reduced3} be given by $\tilde {\theta}_{\rm r}(t)$, and suppose Assumptions \ref{assum:1}--\ref{assum:3} hold. Then there exist positive constants $\mu_1, \mu_2$ and $\omegaf^*$ such that for all 
\begin{equation}
|| {\tilde \theta}^{\rm a}_{\rm c}(0) || < \mu_1, \ \ || z(0) - r({\tilde \theta}^{\rm a}_{\rm c}(0)) || <\mu_2,
\ \ \omegaf > \omegaf^*
\end{equation}
the singular perturbation problem \eqref{eqn:sing_pert1}--\eqref{eqn:sing_pert2} has a unique solution for all $t>0$ and
\begin{equation}\label{eq-SPtheta}
||{\tilde \theta}^{\rm a}(t) - \tilde {\theta}_{\rm r}(t)|| = O (1/\omegaf) \quad \text{for all }t \in [0, \infty).
\end{equation}
\end{proposition}

We omit other bounds given by \cite{khalil} (those based on the quasi-steady state \eqref{eq-qss}) as they do not relate to safety. Proposition \ref{prop:closeness_avg} describes closeness of the original system with that of the average system. Proposition \ref{prop:sing_pert} describes a closeness in the average system and the reduced model \textit{of the average system} \eqref{eqn:reduced2} (with $\omegaf$ large). Now we show the reduced order model \eqref{eqn:reduced3} is indeed safe.

\begin{proposition}[Safety of the Reduced System] \label{prop:safety_of_reduced}
Under the dynamics of the system \eqref{eqn:reduced3}, the set $\{ \tilde \theta_{\rm r} \in \mathbb{R}: h(\tilde \theta_{\rm r} + \theta^*) \geq 0 \}$ is forward invariant and $h(\tilde \theta_{\rm r}(t) + \theta^*) \geq h(\tilde \theta_{\rm r}(0) + \theta^*) e^{-ct}$ for all $t \geq 0$.
\end{proposition}
\begin{proof}
We can verify that \eqref{eqn:reduced3} renders the set forward invariant by showing $\dot{h}(\tilde \theta_{\rm r} + \theta^*) + c h(\tilde \theta_{\rm r} + \theta^*) \geq 0$ \cite{ames2016control}:
\begin{align}\label{eq-hdot}
    \dot{h} + c h &= - k h_1^T H \tilde \theta_{\rm r} + c h(\tilde \theta_{\rm r} + \theta^*)  + \nonumber \\& \qquad \maxapprox\{k h_1^T H \tilde \theta_{\rm r} - c  h(\tilde \theta_{\rm r} + \theta^*)\} > 0 
\end{align}
where we use the relation $h_0 + h_1^T \tilde \theta_{\rm r} = h(\tilde \theta_{\rm r} + \theta^*) $. The inequality holds because for any $x$, $-x + \maxapprox\{x\} > 0$.
\end{proof}
The safety of the reduced system, and its closeness to that of the average system provide the key intermediate results in achieving Theorem \ref{thm:assignable_prac_safety} --- the proof is shown in the appendix.

\section{Newton-Based Assignably Safe Extremum Seeking} \label{sec:NBASfES}
One might imagine that if we can assign a bound on the rate of $h$, while conducting gradient-based optimization, then we should be assign safety when performing Newton-based optimization \cite{ghaffari2012multivariable}, which has an assignable rate of convergence of the parameter itself. A Newton-Based Assignably Safe ES (NB-ASfES) scheme would conceivably achieve 1) assignable safety and 2) assignable convergence of the parameter. In NB-ASfES we hope for a convergence rate of the parameter to be assigned a rate $k$ when the assigned safety condition (with rate $c$) is not violated -- and if the safety condition is violated, the convergence of the parameter is such that the safety rate assignment is maintained. We show in this section that this is \emph{not possible} with the NB-ASfES approach in multiple dimensions, because the NB-ASfES scheme will not solve the optimization problem in \eqref{eqn:opt_problem_intro} in general. Furthermore, we also show that a NB-ASfES scheme is only useful for $n=1$.

To understand the basic dynamics of the ASfES scheme, we can think of the algorithm approximating the following dynamics:
\begin{equation} \label{eqn:exact_safe_scheme}
    \dot \theta = u_0 + \frac{\nabla h(\theta)}{|| \nabla h(\theta) ||^2} \max\{-u_0^T \nabla h(\theta) - c h(\theta), 0 \}  ,
\end{equation}
with $u_0 = -k \nabla J(\theta)$. The form of \eqref{eqn:exact_safe_scheme} guarantees safety and can be derived for a nominal control law $u_0$ using the standard QP formulation given in \cite{ames2016control}. Now consider a NB-ASfES controller, with the nominal control $u_0 = -k H^{-1} \nabla J(\theta) = -k(\theta - \theta^*)$ where $k>0$ is the assigned rate of convergence. Note that in the extremum seeking form of \eqref{eqn:exact_safe_scheme}, the NB-ASfES algorithm contains a state $\Gamma(t)$ which provides the estimate of $H^{-1}$ \cite{ghaffari2012multivariable}. 

Consider the case of $h_0 < 0$ where the minimizer of $J$ on $\{h(\theta) \geq 0 \}$ is unsafe and given by
\begin{equation}
    \theta_{\textup{smin}} = \frac{|h_0| H^{-1} h_1}{h_1^T H^{-1} h_1} + \theta^*.
\end{equation}
See the proof of Proposition \ref{prop:const_min} in the Appendix for this fact. Now let us ask the question: is the equilibrium of \eqref{eqn:exact_safe_scheme} equal to $\theta_{\textup{smin}}$ when $u_0=-k(\theta - \theta^*)$ and $h_0 < 0$ under Assumptions \ref{assum:1} - \ref{assum:3}? Solving for 
\begin{multline}
        -k(\theta_{\textup{smin}} - \theta^*) +\frac{\nabla h(\theta_{\textup{smin}})}{|| \nabla h(\theta_{\textup{smin}} ||^2} \times \\ \max\{k(\theta_{\textup{smin}} - \theta^*)^T \nabla h(\theta_{\textup{smin}}) - c h(\theta_{\textup{smin}}), 0 \} = 0
\end{multline}
yields the condition 
\begin{equation}
    H^{-1} h_1 = \frac{h_1 H^{-1} h_1}{|| h_1 ||^2} h_1. \label{eqn:NB-ASfES_condition}
\end{equation}
This is nothing more than an eigenvector equation stating that $h_1$ must be an eigenvector of $H^{-1}$. Because $H$ and $H^{-1}$ share eigenvectors, $h_1$ must also be an eigenvector of $H$. This is a restrictive condition, stating that the gradient of $h$, must point along one of the principle axes of the ellipsoids formed by the levels of $J$.

One can also come to the conclusion that NB-ASfES will not solve the constrained optimization problem with some geometric intuition for the case $n=2$. Intuition says that ASfES, with the nominal control law of $u_0=-k \nabla J$, causes trajectories to descend down the level curves of $J$, finding the smallest level curve of $J$ which intersects the safe set at some point $p = \theta_{\textup{smin}}$. The level curve $J(\theta) = J(\theta_{\textup{smin}})$ will in general be an ellipse in $2$ dimensions. If the nominal control law $u_0=-k(\theta - \theta^*)$ is used instead, trajectories will now travel down level curves of some other function $\tilde J(\theta) = \frac{1}{2}(\theta - \theta^*)^T (\theta - \theta^*)$ because $(\theta - \theta^*)$ is the gradient of $\tilde J$. The smallest level set of $\tilde J$ touching the safe set is, in general, a circle and will intersect at some other point $\tilde p$, different from the point $p$ given by the ellipse generated from $J$.

We cannot assume that the restrictive condition \eqref{eqn:NB-ASfES_condition} holds when $h_1$ and $H$ are unknown to the user, and therefore we do not present a NB-ASfES scheme for the generic $n$ dimensional case. The only situation where \eqref{eqn:NB-ASfES_condition} holds is in the 1 dimensional case. 

Based on this discussion, we propose a NB-ASfES scheme only for the case of $n=1$, as
\begin{align}
    \dot{\hat \theta} &= -k \Gamma G_J + \gamma \maxapprox\{k \Gamma G_J G_h - c \eta_h\} G_h \label{eqn:th_dyn_NB}\\
    \dot G_J &= -\omegaf G_J + \omegaf(J( \hat \theta(t) + S(t) ) - \eta_J)M(t) \label{eqn:gj_dyn_NB}\\
    \dot \eta_J &= -\omegaf \eta_J + \omegaf J( \hat \theta(t) + S(t) ) \label{eqn:etaj_dyn_NB}\\
    \dot G_h &= -\omegaf G_h + \omegaf(h( \hat \theta(t) + S(t) ) - \eta_h)M(t) \label{eqn:gh_dyn_NB}\\
    \dot \eta_h &= -\omegaf \eta_h + \omegaf h( \hat \theta(t) + S(t) ) \label{eqn:etah_dyn_NB} \\
    \dot \gamma &= \omegaf  \gamma (1 - \gamma |G_h|^2) \label{eqn:gamma_dyn_NB} \\
    \dot \Gamma &= \omegaf  \Gamma (1 - \Gamma J( \hat \theta(t) + S(t) )N(t)) \label{eqn:Gamma_dyn_NB}
\end{align}
where $N(t) = \frac{16}{a^2}(\sin^2(\omega t) -\frac{1}{2})$. 

The resulting theory is identical to that of Theorem \ref{thm:convergence_avg} and \ref{thm:assignable_prac_safety} for the scalar algorithm given in \eqref{eqn:th_dyn_NB} - \eqref{eqn:Gamma_dyn_NB}, and follow similar analysis steps (but for $n=1$) shown in Section \ref{sec:convergence} and Section \ref{sec:safety_in_the_average}. The scalar NB-ASfES algorithm shares the assignable safety property with that of the gradient-based algorithm ASfES, with the additional feature that the nominal convergence rate of the parameter can also be assigned.

In Section \ref{sec:scalar} we illustrate the behavior of the NB-ASfES algorithm. 
\begin{figure*}[t]
  \centering
  \includegraphics[scale=0.75]{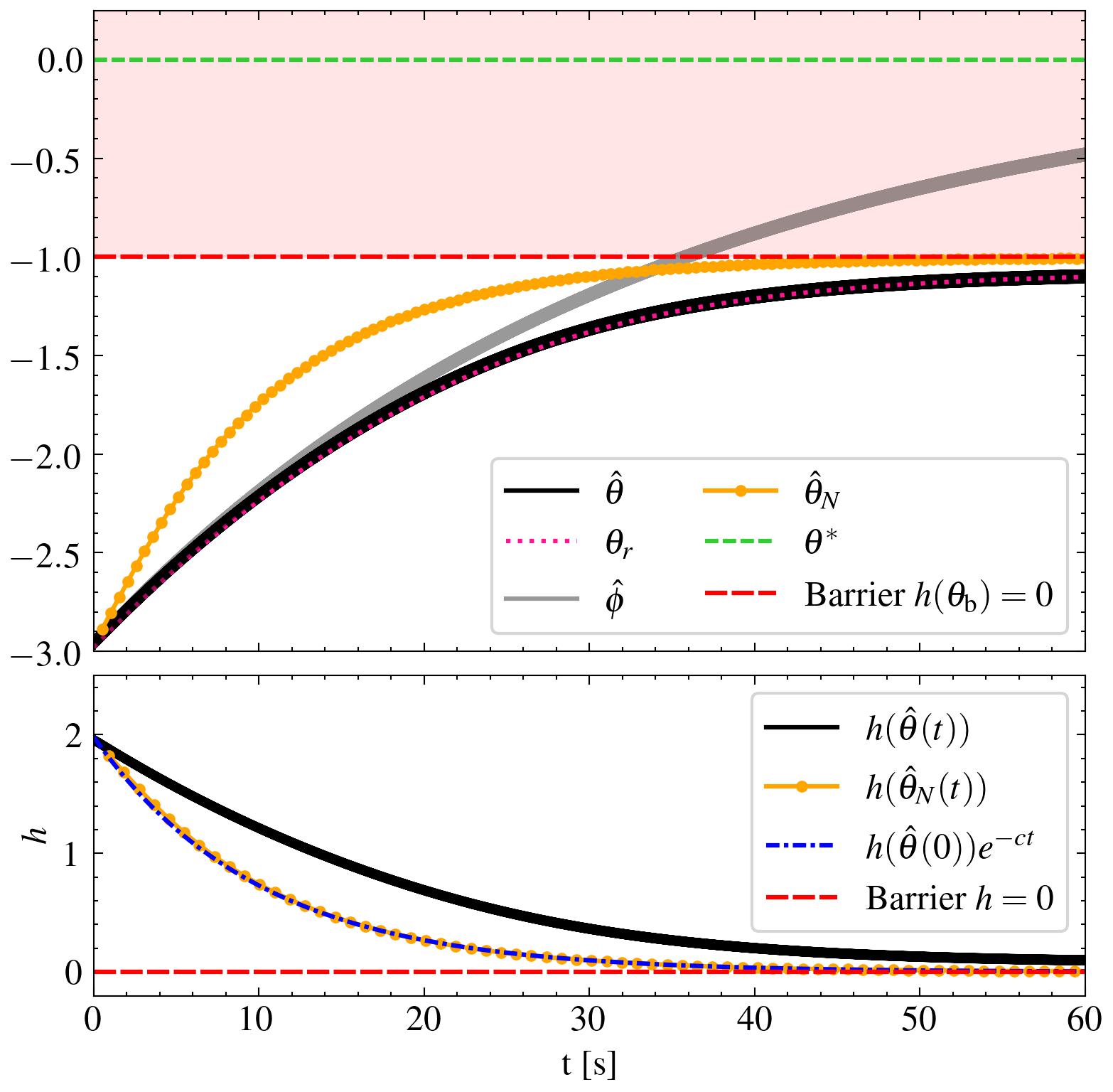}
  \caption{A demonstration of the algorithm in 1D with the optimizer lying in the unsafe region --- trajectories of the original and reduced system of ASfES are given by  $\hat \theta$ and $\theta_{r}$ and the NB-ASfES trajectory is given by $\hat \theta_{N}$. The classical ES solution is given by $\hat \phi$, and the red shaded region marks the unsafe set where $h<0$.}
  \label{fig:scalar}
\end{figure*}

\section{Simulations} \label{sec:simulations}
\subsection{Example 1: Scalar System} \label{sec:scalar}
To demonstrate the closeness of trajectories provided by our analysis, we show the solutions of the original ASfES algorithm \eqref{eqn:th_dyn}--\eqref{eqn:gamma_dyn} and reduced system \eqref{eqn:reduced2}, under Assumptions \ref{assum:1}--\ref{assum:3} in Figure \ref{fig:scalar}. The unknown function parameters are $\theta^* = 0$ (so $\theta = \tilde \theta$), $J^* = 0$, $H = 0.1$, $h_0 = -1$ and $h_1 = -1$. The ASfES and NB-ASfES controller parameters are $a = 0.25$, $k = 0.3$, $c=0.1$, $\delta = 10^{-3}$, $\omega= 200$, and $\omegaf = 3$. The initial conditions are $\hat \theta(0)=\hat \theta_{N}(0)=\theta_{r}(0) = -3$, while all other estimator states are initialized to their exact quantities. The trajectories of the ASfES and NB-ASfES algorithm parameters are denoted by $\hat \theta(t)$ and $\hat \theta_{N}(t)$ respectively. Note that although we do not know the gradients of $J$ and $h$ to initialize $G_J$ and $G_h$ to exactly $\nabla J(\theta(0))$ and $\nabla h(\theta(0))$, we can generate an accurate estimate by ``warming up'' the algorithm in \eqref{eqn:th_dyn}--\eqref{eqn:gamma_dyn} by setting $\dot{\hat{\theta}}=0$, and integrating over time until the states $G_{J}, \eta_J, G_h, \eta_h \gamma$ (and $\Gamma$ for NB-ASfES) converge adequately.

We compare the ASfES and NB-ASfES algorithms with the classical ES algorithm denoted by solution $\phi(t)$ which obey the dynamics
\begin{align*}
    \dot{\hat \phi} &= -k G_J \\
    \dot G_J &= -\omegaf G_J + \omegaf(J( \hat \phi(t) + S(t) ) - \eta_J)M(t)  \\
    \dot \eta_J &= -\omegaf \eta_J + \omegaf J( \hat \phi(t) + S(t) ) 
\end{align*}
with $\phi = \hat \phi + S(t)$. The initial condition is $\phi(0) = -3$ and other filtered estimates are initialized to their exact quantities. The gains $a, k , \omegaf, \omega$ are set the same as in the ASfES and NB-ASfES schemes.

The reduced system in the upper part of Fig. \ref{fig:scalar} maintains safety while the ASfES and NB-ASfES schemes maintain practical safety. In the lower part of Fig. \ref{fig:scalar} we observe the behavior of $h(\hat \theta(t))$ and $h(\hat \theta_{N}(t))$ have the assigned time constant $c$ bound associated with its approach to the barrier. The NB-ASfES exhibits faster convergence due to the cancellation of the small Hessian, $H=0.1$, yet maintains the assigned practical safety bound, which can be observed by the closeness between the plots of $h(\hat \theta_{N}(t))$ and $h(\hat \theta(0)) e^{-ct}$. 

\subsection{Example 2: Quadratic $J$ and Linear $h$ in 2 Dimensions}
\begin{figure*}[t]
  \centering
  \includegraphics[scale=0.75]{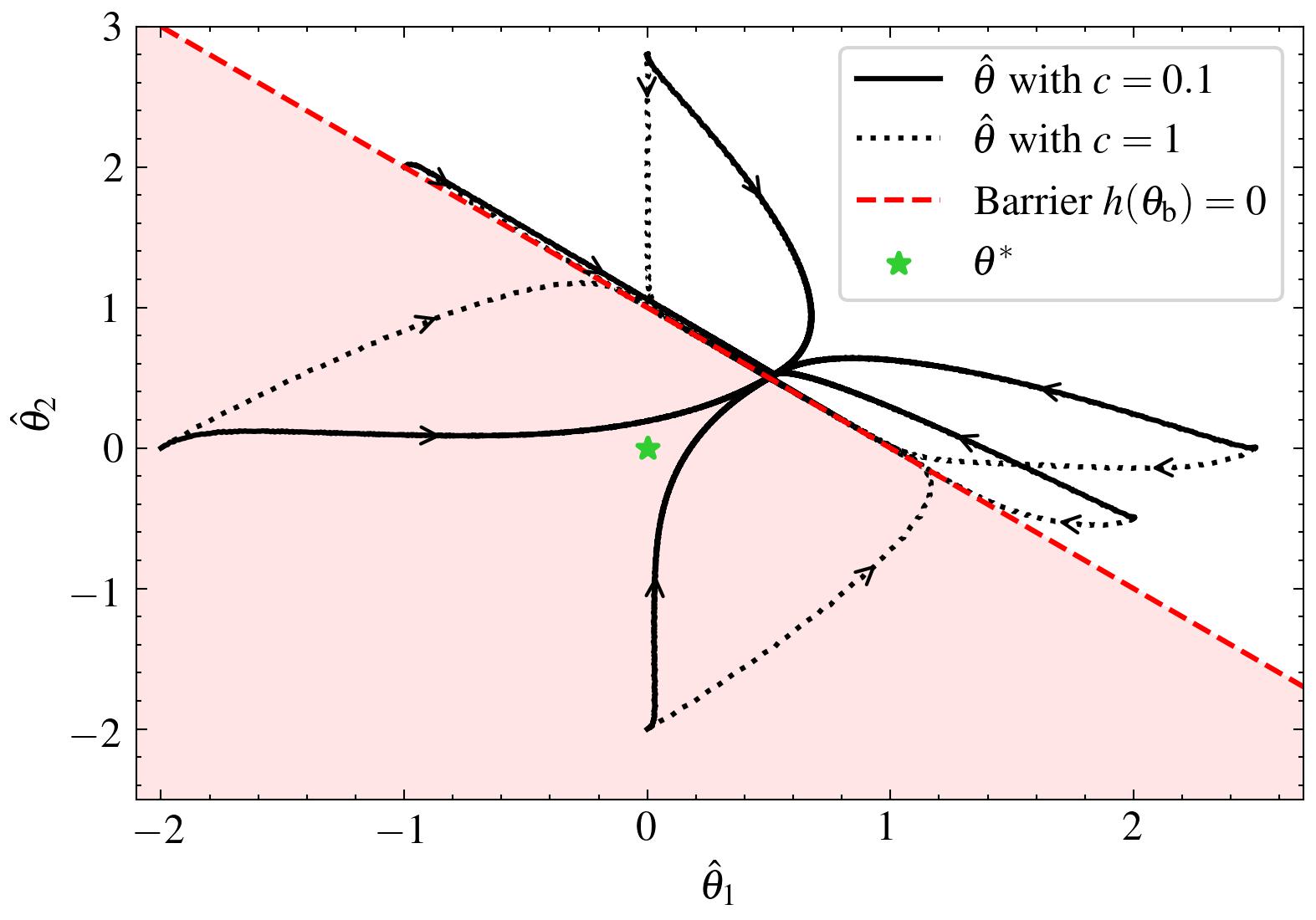}
  \caption{Trajectories of $\hat \theta$ in 2D with a linear $h$ and quadratic $J$. Six initial conditions, starting both inside and outside the safe set, and their trajectories are plotted for both $c=1$ and $c=0.1$. The global optimizer $\theta^*$ is outside of the safe set.}
  \label{fig:2d_linear_h}
\end{figure*}
In this final example we illustrate the effect of $c$ on ASfES in 2 dimensions. The unknown functions are $h = \theta_1 + \theta_2 - 1$ and $J = \theta_1^2 + \theta_2^2$. The controller parameters are $a = 0.25$, $k = 0.1$, $\delta = 10^{-3}$, $\omega_1 = 75$, $\omega_2 = 100$, and $\omegaf = 3$. Trajectories are plotted in Fig. \ref{fig:2d_linear_h} for various initial conditions for $c=1$ and $c=0.1$ with all other controller states are initialized to their exact quantities.

Notice that a larger value of $c$ allows a direct approach  (and a truly faster approach in time, see Theorem \ref{thm:assignable_prac_safety}) to the unsafe region when starting from the safe region but a more direct approach to the safe region when starting from the unsafe region.

\section{Conclusion}
We present ASfES and NB-ASfES which achieve an assignable safety attractivity rate, specifying a bound on the rate at which $h$, the measurement of safety, is allowed to decay --- with trajectories moving towards the safe set from the unsafe region, or moving towards the unsafe set from a safe region. We demonstrate that the NB-ASfES scheme is only useful in the case of a scalar parameter, and may lead to non-optimal convergence if used in multiple dimensions, yet features an assignable rate of parameter convergence which is favored only if it agrees with the safety attractivity rate. 

The algorithms presented here are local in nature, but provide benefits over their semiglobal versions, with assignable safety attractivity and assignable (nominal) parameter convergence.

\section*{Appendix}
This section includes the analysis of the matrix \eqref{J11_} and the proof of Lemma \ref{corol:hurwitz}. It also includes additional details on the spectrum of matrices of this form. The first result describes a general algebraic relationship between the spectra of two matrices, $X$ and $Z$. The second result show the eigenvalues of a particular $Z$ are real and positive. These two results and Proposition \ref{prop:const_min} are used in Theorem \ref{thm:convergence_avg}. We then present the proof of Theorem \ref{thm:assignable_prac_safety}. We use the notation of a unit vector $\hat{h}_1 = h_1  || h_1||^{-1}$.
\subsection{Eigenvalues of $X$ and $Z$}
\begin{proposition} \label{prop:eig_relation}
Consider the $\mathbb{R}^{2n+1 \times 2n+1}$ matrix
\begin{equation*}
X = \begin{bmatrix}
    0 && -M && - \tilde{c}h_1  \\
    \omegaf H && - \omegaf I  && 0 \\
     \omegaf h_1^T && 0 && -\omegaf
    \end{bmatrix}
\end{equation*}
with $\omegaf, \tilde{c} > 0$, $h_1 \in \mathbb{R}^n$ and matrices $M, H \in \mathbb{R}^{n \times n}$.

Then, for all $\lambda \in \sigma(X)$ such that $\lambda \neq -\omegaf$, there exists a $\bar{\lambda} \in \sigma(\omegaf M H +\omegaf \tilde{c}h_1 h_1 ^T)$ satisfying $\lambda^2 + \omegaf \lambda + \bar \lambda =0$.
\end{proposition}
\begin{proof}
Any eigenvalue $\lambda \neq -\omegaf$, of the matrix $X$, must satisfy the eigenvalue expression for a nonzero eigenvector,
\begin{equation}
    \begin{bmatrix}
    0 && -M && - \tilde{c}h_1  \\
    \omegaf H && - \omegaf I  && 0 \\
    \omegaf h_1^T && 0 && -\omegaf
    \end{bmatrix}
    \begin{pmatrix}
    u  \\
    v \\
    w
    \end{pmatrix}
    = \lambda
    \begin{pmatrix}
    u  \\
    v \\
    w
    \end{pmatrix} \label{eqn:eig_exp}
\end{equation}
with $u, v \in \mathbb{C}^n$, $w \in \mathbb{C}$. From the second and third rows of \eqref{eqn:eig_exp} we have
\begin{align}
    \omegaf H u &= (\lambda + \omegaf) v \label{eqn:eigr1}\\
    \omegaf h_1^T u &= (\lambda + \omegaf) w \label{eqn:eigr2}
\end{align}
and from the first row of \eqref{eqn:eig_exp} we have
\begin{align}
    -(\lambda + \omegaf) M v - (\lambda + \omegaf) \tilde{c} h_1 w &= (\lambda + \omegaf)\lambda u, \label{eqn:eigr3}
\end{align}
where we have multiplied on both sides $\lambda + \omegaf \neq 0$. Now we substitute expressions in \eqref{eqn:eigr3} for $(\lambda + \omegaf)w$ and $(\lambda + \omegaf) v$,
\begin{align}
    (\omegaf M H +\omegaf \tilde{c}h_1 h_1 ^T)u = -(\lambda^2 + \omegaf \lambda) u \label{eqn:eig_eqn_Z}.
\end{align}
What we arrive at, is another eigenvalue equation describing the scaling of the vector $u$ by the value $-(\lambda^2 + \omegaf \lambda)$.

It is also guaranteed that $u \neq 0$ for the eigenvector equation \eqref{eqn:eig_exp} under the following argument: suppose not and $u=0$, which implies in that $v=0$ and $w=0$ from \eqref{eqn:eigr1} and \eqref{eqn:eigr2} because $\lambda +\omegaf \neq 0$. Because the equation \eqref{eqn:eig_exp} must have nonzero eigenvector then we have a contraction which implies $u \neq 0$. 

Therefore we can say that $\bar{\lambda}$ is an eigenvalue of $Z$ defined as
\begin{equation}
    Z:=\omegaf M H +\omegaf \tilde{c}h_1 h_1 ^T,
\end{equation}
where
\begin{equation}
     \bar{\lambda} = - (\lambda^2 + \omegaf \lambda).
\end{equation} \label{eq:eigenval}
\end{proof}

\subsection{Eigenvalues of $Z$ are real and positive}
\begin{lemma} \label{lem:eig_real_pos}
Consider the matrix $Z:=\omegaf M H +\omegaf \tilde{c}h_1 h_1 ^T$ with the following parameters: $H \succ 0$, $M = k \left( I -  \alpha  \frac{h_1 h_1^T}{|| h_1 ||^2} \right)$, $k, \omegaf >0$, $\tilde{c} \geq 0$, $0 \leq \alpha \leq 1 $, and $h_1 \in \mathbb{R}^{n}_{\neq 0}$ such that if the scalar $\alpha=1$ then $\tilde{c} >0$ strictly.
Then, each $\bar{\lambda} \in \sigma(Z)$ is real and strictly positive.
\end{lemma}
\begin{proof}
We consider three cases of $\alpha$ in this proof. Case 1 is when $\alpha=0$, Case 2 is $\alpha \in (0,1)$, and Case 3 is when $\alpha = 1$.

\textbf{Case 1:} If $\alpha = 0$, the matrix $Z$ can readily be shown to have real and strictly positive eigenvalues because it is the sum of a positive definite matrix and a positive semi definite matrix.

\textbf{Case 2:} Consider $\alpha \in (0,1)$. The positive definite matrix $M \succ 0$ is recalled below and its' inverse is also shown:
\begin{align}
    M &= k (I - \alpha \hat{h}_1 \hat{h}_1^T) \succ 0, \\
    M^{-1} &= k^{-1} \left( I + \frac{\alpha}{1-\alpha} \hat{h}_1 \hat{h}_1^T \right) \succ 0 \label{eqn:m_inv},
\end{align}
where $\hat{h}_1 = h_1  || h_1||^{-1}$. The reader can check by hand that $M M^{-1} = M^{-1} M = I$ holds, see \cite{miller1981inverse} for background. We use \eqref{eqn:m_inv} to derive the expression  
\begin{align}
     \hat{h}_1 \hat{h}_1^T &= \frac{(1-\alpha)}{\alpha} k M^{-1} - \frac{(1-\alpha)}{\alpha} I \quad \text{for } \alpha \neq 0 ,
\end{align}
and rewrite $Z$ as
\begin{align}
    Z &= \omegaf M H +\omegaf \tilde{c}h_1 h_1 ^T,  \\
    Z &= \omegaf M H + \beta \left( (1-\alpha)  k M^{-1} - (1-\alpha) I \right).
\end{align}
for some positive scalar $\beta = \omegaf \tilde{c} / \alpha \geq 0$. Now we perform a similarity transformation $Z' = T^{-1} Z T $. Let us take $T = M^{1/2} \succ 0$ and compute
\begin{align}
Z' = M^{-1/2} Z M^{1/2} = \omegaf M^{1/2} H M^{1/2}+ \beta \alpha \hat{h}_1 \hat{h}_1^T.
\end{align}
The matrix $Z'$ is the sum of a positive definite matrix $\omegaf M^{1/2} H M^{1/2}$ (which can be shown by definition) and a positive semi definite matrix $\beta \alpha \hat{h}_1 \hat{h}_1^T$. Because two similar matrices have the same eigenvalues, it follows that $Z$ shares its' eigenvalues with that of a symmetric, positive definite matrix $Z'$ which is known to have real and positive eigenvalues. 

\textbf{Case 3: } Let us consider the case of $\alpha = 1$, $\tilde{c} > 0$ and rewrite $Z$ as 
\begin{align}
    Z &= \omegaf k \left( I -  \frac{h_1 h_1^T}{|| h_1 ||^2} \right) H +\omegaf \tilde{c}h_1 h_1 ^T \\
    &= \omegaf k \left( I -  \hat{h}_1 \hat{h}_1 \right) H +\mu \hat{h}_1 \hat{h}_1 ^T
\end{align}
with $\mu = \omegaf \tilde{c}||h_1 ||^2 > 0$ .The symmetric, rank 1 matrix $\hat{h}_1 \hat{h}_1$, with spectrum $\sigma(\hat{h}_1 \hat{h}_1) = \{1,0,0,...,0 \}$, is diagonalizable by an orthogonal matrix $U$ such that $U^T = U^{-1}$. We can write
\begin{equation}
    \hat{h}_1 \hat{h}_1^T = U D U^T = U 
    \begin{bmatrix}
    1 & 0_{1 \times n-1} \\
    0_{n-1 \times 1} & 0_{n-1 \times n-1}
    \end{bmatrix} 
    U^T
\end{equation}
with $D$ having a single nonzero element $d_{11} = 1$. Then it follows that
\begin{align}
    Z &= \omegaf k (I - U D U^T)H + \mu U D U^T \nonumber\\
    &= U \left( \omegaf k (I - D )U^T H U + \mu D \right) U^T \nonumber\\
    &= U \left( \omegaf k (I - D )\tilde{H} + \mu D \right) U^T 
\end{align}
where $\tilde{H} \equiv U^T H U \succ 0$ is a matrix similar to $H$. Therefore $Z$ is similar to $Z' \equiv \omegaf k (I - D )\tilde{H} + \mu D$ which can be written as block triangular:
\begin{align}
Z' &= \omegaf k (I - D )\tilde{H} + \mu D \nonumber\\
    &=
    \omegaf k
    \begin{bmatrix}
    0 && 0   \\
    0 && I
    \end{bmatrix}
    \begin{bmatrix}
    \tilde{h}_{11} && \tilde{H}_{12}^T   \\
    \tilde{H}_{12} && \tilde{H}_{22}
    \end{bmatrix}
    +
    \mu
    \begin{bmatrix}
    1 && 0   \\
    0 && 0
    \end{bmatrix} \nonumber \\
    &=
    \begin{bmatrix}
    \mu && 0   \\
    \omegaf k \tilde{H}_{12} && \omegaf k \tilde{H}_{22}
    \end{bmatrix}.    \label{eqn:z_c3_1}
\end{align}
The matrix $\tilde{H}_{22}$ is a diagonal block of a positive definite matrix, and is positive definite itself \cite{horn2012matrix}. Because $Z'$ is block triangular with positive definite diagonal blocks, $Z$ has real and positive eigenvalues.
\end{proof}
\subsection{Properties of $X$}
In the next lemma, we extend the conclusions from the prior two results in this section and make more specific claims on the eigenvalues of the linearizing matrix. Statement \ref{stmt:spectrum} assigns each eigenvalue of $Z$ to a pair of eigenvalues of $X$. This is an extension of the results in Proposition \ref{prop:eig_relation}, which only concerns the existence of an eigenvalue of $Z$. 
\begin{lemma} \label{lem:all_properties_X}
Under Assumptions \ref{assum:1} - \ref{assum:3}, consider
\begin{equation}
X =
    \begin{bmatrix}
    0 && -M && - \tilde{c} h_1 \\
    \omegaf H && -\omegaf I  && 0 \\
    \omegaf h_1^T && 0 && -\omegaf
    \end{bmatrix}, \label{X}
\end{equation}
with
$0 \leq \alpha \leq 1 $, $ M = k \left( I -  \alpha \hat{h}_1 \hat{h}_1^T \right)$, and $\tilde{c}=c \alpha|| h_1 ||^{-2}$. Then
\begin{enumerate}
    \item X is Hurwtiz. \label{stmnt:hurwtiz}
    \item The spectrum $\sigma (X) = \{-\omegaf, \lambda_1^+, \lambda_1^-, ..., \lambda_n^+, \lambda_n^- \}$ where $\lambda_i^+ / \lambda_i^- \neq -\omegaf$ are the positive and negative solutions to $\lambda_i^2 + \omegaf \lambda_i + \bar{\lambda}_i=0$ for each $\bar{\lambda}_i \in \sigma (Z)$ and for all $\alpha \in [0,1]$, where $Z:=\omegaf M H +\omegaf \tilde{c}h_1 h_1 ^T$. \label{stmt:spectrum}
    \begin{enumerate}
        \item If $\alpha = 0$ then $\sigma (Z) = \sigma(\omegaf k H)$. \label{stmnt:alpha0}
        \item If $\alpha = 1$ then $\sigma (Z) = \{\omegaf c \} \cup \sigma(\omegaf k \Tilde{H}_{22})$ where $\tilde{H}_{22} \succ 0$ is the lower right $n-1 \times n-1$ block of $\tilde{H} = U^T H U $ for an orthogonal matrix $U$ which diagonalizes $\hat{h}_1 \hat{h}_1^T$.\label{stmnt:alpha1}
    \end{enumerate}
\end{enumerate}
\end{lemma}

\begin{proof}

Statement \ref{stmnt:hurwtiz} follows directly from Proposition \ref{prop:eig_relation} and Lemma \ref{lem:eig_real_pos}. From Proposition \ref{prop:eig_relation}: each eigenvalue $\lambda \in \sigma(X)$, for which $\lambda \neq -\omegaf$, there exists a $\bar{\lambda} \in \sigma(\omegaf M H +\omegaf \tilde{c}h_1 h_1 ^T)$ satisfying $\lambda^2 + \omegaf \lambda + \bar \lambda =0$. Because $\bar \lambda>0$ (Lemma \ref{lem:eig_real_pos}), by the Routh-Hurwtiz criterion $X$ is Hurwtiz. 

Statement \ref{stmt:spectrum} can be shown by taking an eigenvalue of $Z$, denoted as the positive scalar $\bar \lambda > 0$ (positivity is shown in Lemma \ref{lem:eig_real_pos}) and showing that it generates 2 eigenvalues of $X$. Take the logic and algebra of Proposition \ref{prop:eig_relation} in reverse. Express $\bar \lambda = -(\lambda_i^2 + \omegaf \lambda_i)$ for either $\lambda_i = \lambda_i^+$ or $\lambda_i =\lambda_i^-$ satisfying $\lambda_i^2 + \omegaf \lambda_i + \bar \lambda= 0$. The eigenvalue equation is
\begin{align}
    (\omegaf M H +\omegaf \tilde{c}h_1 h_1 ^T)u = -(\lambda_i^2 + \omegaf \lambda_i) u \label{eqn:eig_eqn_Z2}.
\end{align}
for some distinct eigenvector $u \in \mathbb{C}^n_{\neq 0}$. We introduce $v \in \mathbb{C}^n, w \in \mathbb{C}$ satisfying 
\begin{align}
    \omegaf H u &= (\lambda_i + \omegaf) v ,\label{eqn:eigr1_2}\\
    \omegaf h_1^T u &= (\lambda_i + \omegaf) w ,\label{eqn:eigr2_2}
\end{align}
and make substitutions to get
\begin{align}
    -(\lambda_i + \omegaf) M v - (\lambda_i + \omegaf) \tilde{c} h_1 w &= (\lambda_i + \omegaf)\lambda_i u. \label{eqn:eigr3_2}
\end{align}
where $\lambda_i \neq - \omegaf$ because $\lambda_i^2 + \omegaf \lambda_i + \bar \lambda= 0$ and $\bar \lambda > 0$. Then we recover the eigenvector equation in \eqref{eqn:eig_exp} with the eigenvalue $\lambda_i$ satisfying $X \nu = \lambda_i \nu$ with $\nu = [u^T, v^T, w]^T$. Because this holds for $2n$ eigenvalues ($\lambda_i=\lambda_i^+$ and $\lambda_i=\lambda_i^-$ for each $\bar \lambda \neq -\omegaf$), we have accounted for all but one of the $2n + 1$ eigenvalues in $\sigma (X)$.

Finally, the spectrum of $X$ always contains an eigenvalue $-\omegaf$, the last unaccounted eigenvalue. When $\alpha = 0$, the matrix $Y$ has an eigenvalue of $-\omegaf$, due to the triangular structure. For $\alpha \in (0,1]$, the determinant $\det(X+\omegaf I)$ is
\begin{align}
\left|X+\omegaf I\right| &= 
\left|
\arraycolsep=0.8pt\def\arraystretch{0.8} 
\begin{array}{c c c }
    I \omegaf & -M & - \tilde{c}h_1  \\
    \omegaf H & 0  & 0 \\
    \omegaf h_1^T & 0 & 0
\end{array} \right| 
=\omegaf^n
\left| 
\arraycolsep=1.5pt\def\arraystretch{1.0} 
\begin{array}{c c}
    H M & \Tilde{c} H h_1 \\
    h_1^T M & \tilde{c} || h_1||^2
\end{array} \right| \nonumber\\
&= \tilde{c} || h_1||^2 \omegaf^n \det(H M - \tilde{c} H h_1 (\tilde{c} || h_1||^2)^{-1} h_1^T M) \nonumber\\
&=\tilde{c} || h_1||^2 \omegaf^n \det(H (I -  \hat{h}_1 \hat{h}_1^T) M) = 0
\end{align}
using rules for determinants of block matrices. Because $(I - \hat{h_1} \hat{h_1}^T)$ is singular and the determinant of a product is the product of the determinants, then $\det(X+\omegaf I)=0$ and $-\omegaf$ is an eigenvalue. 

Statement \ref{stmnt:alpha0} can be shown by noting that $Z$ for $\alpha=0$, simply reduces to $k \omegaf H$.

Statement \ref{stmnt:alpha1} can be shown by looking at the analysis of this case in the proof in Lemma \ref{lem:eig_real_pos}, Case 3. The matrix $Z$ is similar to the block diagonal matrix $Z'$ which has spectra of its diagonal blocks: $\sigma(\mu) \cup \sigma(\omegaf k \tilde{H}_{22})$. The matrix $\tilde{H}_{22} \succ 0$ is the lower right $n-1 \times n-1$ block of $\tilde{H} = U^T H U $ for a orthogonal matrix $U$ which diagonalizes $\hat{h}_1 \hat{h}_1^T$. Finally, making the substitutions for $\mu$ and $\Tilde{c}$, Statement \ref{stmnt:alpha1} is shown.
\end{proof}

\subsection{Constrained Minimum}
\begin{proposition}
    The minimum of $J$ on the set $\mathcal{C} = \{\theta: h(\theta) \geq 0 \}$, for $J$ and $h$ given in Assumptions \ref{assum:1}-\ref{assum:2}, is $$J_s^* = J^* + \frac{h_0^2}{2 h_1^T H^{-1} h_1} u(-h_0)$$
    where $u$ is the unit step function. \label{prop:const_min}
\end{proposition}
\begin{proof}
We use the following convex analysis result: ``A point $\theta_{\textup{smin}}$ is the minimum of $J$ relative to $S$ if and only if $-\nabla J(x)$ is normal to $S$ at $\theta_{\textup{smin}}$". This result is restated from Theorem 27.4 \cite{rockafellar1970convex}. Additionally, one can find this result in \cite{boyd2004convex} Section 4.2.3.

When $h_0\geq0$ the minimum of $J$ on $S$ is simply $J^*$. 

When $h_0 <0$, we use the normality condition to yield
\begin{equation}
    p h_1 = H(\theta_{\textup{smin}} - \theta^*)
\end{equation}
for some $p$ using the fact that $\nabla J(\theta) = H(\theta - \theta^*)$ and $\nabla h(\theta) = h_1$. Also,
\begin{equation}
    h(\theta_{\textup{smin}}) = h_0 + h_1^T(\theta_{\textup{smin}} - \theta^*)=0.
\end{equation}
Solving, we get
\begin{equation}
    p = -\frac{h_0}{h_1^T H^{-1} h_1}.
\end{equation}
The value of $J$ at $\theta_{\textup{smin}}$ is
\begin{align}
    J(\theta_{\textup{smin}}) &= J^* + \frac{1}{2}(\theta_{\textup{smin}} - \theta^*)^T H (\theta_{\textup{smin}} - \theta^*) \nonumber\\
    &= J^* + \frac{p^2}{2}h_1^T H^{-1} h_1 
    = J^* + \frac{h_0^2}{2 h_1^T H^{-1} h_1}.    
\end{align}
Combining cases $h \geq 0$ and $h < 0$ we reach the result.
\end{proof}
\subsection{Proof Theorem \ref{thm:convergence_avg}}
\begin{proof}
From $\theta(t)-\theta^* =  \tilde\theta(t) +S(t) = \tilde\theta(t) +O(a)$ we have
\begin{equation}
    J(\theta(t)) = J^* + \frac{1}{2}\tilde\theta(t)^T H \tilde\theta(t) + O(a).
\end{equation}
From Proposition \ref{prop:closeness_avg}, because $\|\tilde\theta(t) - \tilde\theta^{\rm a}(t) \| = O(1/\omega)$ and $\tilde\theta^{\rm a}(t) \to \tilde\theta^{\rm a, e}$, then
\begin{equation}
    \limsup_{t\rightarrow\infty} J(\theta(t))= J^* + \frac{1}{2}\tilde\theta^{{\rm a, e}T} H \tilde\theta^{\rm a, e} + O(a)+O(1/\omega). \label{lim:J}
\end{equation}
Before expanding $\theta^{\rm a, e}$ above, we use the Taylor series approximation for small $\delta$ to derive $\sqrt{1 + r \delta} = 1 + O(\delta)$ (for constant $r$) and derive the expression
\begin{align}
    (-\sgn (h_0) + \sqrt{1 + r \delta}) &= 1 - \sgn(h_0) + O(\delta) \nonumber \\
    &= 2(1 - u(h_0)) + O(\delta) \nonumber \\
    &=2 u(-h_0) + O(\delta)
\end{align}
where $u$ is the unit step function. We can write 
\begin{equation}
   \theta^{\rm a, e} = \frac{\|h_0 \| u(-h_0)}{h_1^T H^{-1} h_1}H^{-1} h_1 + O(\delta)
\end{equation}
from the approximation above and \eqref{eqn:eqm_theta}. Expanding part of \eqref{lim:J} we have
\begin{align}
    \frac{1}{2}\tilde\theta^{{\rm a, e}T} H \tilde\theta^{\rm a, e} &= \frac{h_0^2}{2 h_1^T H^{-1} h_1} u(-h_0) + O(\delta).
\end{align}
Using Proposition \ref{prop:const_min}, we achieve the result.
\end{proof}
\subsection{Proof of Theorem \ref{thm:assignable_prac_safety}}
\begin{proof}
From Proposition \ref{prop:closeness_avg} we have $||\theta(t) - (\tilde{\theta}^{\rm a}(t) + \theta^*)|| = O (1/\omega+a)$ for all $t\geq 0$. From Proposition \ref{prop:sing_pert} we have $||{\tilde \theta}^{\rm a}(t) - \tilde {\theta}_{\rm r}(t)|| = O (1/\omegaf)$ for all $t\geq 0$. Therefore $||\theta(t) - (\tilde {\theta}_{\rm r}(t) + \theta^*)|| = O (1/\omegaf +1/\omega+a)$ for all $t\geq 0$. Because $h$ is Lipschitz, $||h(\theta(t)) - h(\tilde {\theta}_{\rm r}(t) + \theta^*)|| = O (1/\omegaf +1/\omega+a)$ for all $t\geq 0$. Using Proposition \ref{prop:safety_of_reduced} we have the result.
\end{proof}

\bibliographystyle{plain}
\bibliography{references}

\end{document}